\documentclass[reqno,12pt]{amsart}
\usepackage[utf8]{inputenc}
\usepackage[T1]{fontenc}
\usepackage{lmodern}
\usepackage{fancyhdr}
\usepackage{amsmath,amssymb,amsfonts}
\usepackage{bm}

\newcommand{\reset}{\setcounter{equation}{0}}

\usepackage{graphicx}
\graphicspath{{figures/}}
\usepackage{subcaption}
\usepackage{xcolor}

\usepackage{array}
\usepackage{longtable}

\usepackage[colorlinks=true,linkcolor=black,citecolor=red,urlcolor=blue]{hyperref}
\usepackage{cite}

\usepackage{enumitem}
\usepackage{float}
\usepackage{caption}

\AtBeginDocument{%
	\setlength{\abovedisplayskip}{8pt plus 2pt minus 1pt}%
	\setlength{\belowdisplayskip}{8pt plus 2pt minus 1pt}%
	\setlength{\abovedisplayshortskip}{5pt plus 2pt minus 1pt}%
	\setlength{\belowdisplayshortskip}{7pt plus 2pt minus 1pt}%
}
\theoremstyle{plain}  
\newtheorem{theorem}{Theorem}[section]
\newtheorem{lemma}[theorem]{Lemma}

\newtheorem{conjecture}[theorem]{Conjecture}

\theoremstyle{definition}

\newtheorem{problem}{Problem}[section]

\newtheorem{remark}{Remark}[section]

\theoremstyle{remark}

\makeatletter
\renewcommand{\thm@space@setup}{%
	\thm@preskip=9pt plus 2pt minus 1pt
	\thm@postskip=9pt plus 2pt minus 1pt}
\makeatother

\title{Concavity Properties of Robin Solutions on $C^{3,1}$ Uniformly Convex Domains}
\author{Dong Ye}
\author{Dekai Zhang}
\address{School of Mathematical Sciences\\ Key Laboratory of Mathematics and Engineering Applications (Ministry of Education)\\Shanghai Key Laboratory of PMMP \\East China Normal University, Shanghai 200241, China \\
	Email: dye@math.ecnu.edu.cn
}
\address{School of Mathematical Sciences\\ Key Laboratory of Mathematics and Engineering Applications (Ministry of Education) \\Shanghai Key Laboratory of PMMP\\East China Normal University, Shanghai 200241, China\\
	Email: dkzhang@math.ecnu.edu.cn}
	\makeatletter
	\renewcommand{\@settitle}{%
		\begin{center}
			\baselineskip14\p@\relax
			\bfseries
			\@title
		\end{center}
	}
	\makeatother
	\makeatletter
	\renewcommand{\section}{\@startsection{section}{1}%
		\z@{.7\linespacing\@plus\linespacing}{.5\linespacing}%
		{\normalfont\bfseries\centering}}
	\makeatother
\begin{document}
	\pagestyle{plain}
	\begin{abstract}
		We prove that, on a bounded uniformly convex domain of class
		$C^{3,1}$, the first Robin eigenfunction is strictly log-concave and the
		Robin torsion function is strictly $\frac12$-concave for all sufficiently
		large Robin parameters. This means that Conjecture~1.1 of
		Andrews, Clutterbuck and Hauer in \cite{AndrewsClutterbuckHauer2020} holds for uniformly convex $C^{3,1}$ domains in any dimension, and settles 
		an open
		 problem posed by Crasta and Fragal\`a \cite{CrastaFragala2021} when the domain has $C^{3,1}$ regularity. Our proof derives suitable uniform
$C^2$-decay estimates by maximum principle arguments, thereby
removing the higher-order boundary regularity assumption in \cite{CrastaFragala2021}.
	\end{abstract}
\vspace*{-8mm}
	\maketitle
\vspace*{-6mm}
		\noindent\textbf{Keywords.} Robin problems; first eigenfunction; torsion function; concavity.

	\noindent\textbf{2020 MSC.}
	 35E10, 35B40, 35B45, 35J25.
	\par\medskip
	
	\section{Introduction}

The convexity of level sets of solutions to elliptic equations is a classical problem in partial differential equations. For Dirichlet problems in convex domains, this question has a long history.  Stronger notions include log-concavity and $p$-concavity. For the Laplacian, two basic examples are the torsion function and the first Dirichlet eigenfunction. Brascamp and Lieb \cite{BrascampLieb1976} proved the log-concavity of the first Dirichlet eigenfunction. For the torsion problem, Makar-Limanov \cite{MakarLimanov1971} proved the $\frac12$-concavity in dimension two, and Kennington \cite{Kennington1985} extended it to arbitrary dimension. 
Both problems were studied for the $2$-Hessian equation, see  \cite{MaXu2008, LiuMaXu2010, Salani2012,BianchiniSalani2013}  in $\mathbb{R}^3$. Very recently,  Li, Ma, Qiu and Salani
\cite{LiMaQiuSalani2026} extended these results for $2$-Hessian in higher dimensions
$n\geq4$.

There are two main approaches to convexity problems of this type. The macroscopic convexity principle is based on the maximum principle  for suitable two-point functions; see Caffarelli and Spruck \cite{CaffarelliSpruck1982}, Korevaar \cite{Korevaar1983}, and Kennington \cite{Kennington1985}. The microscopic convexity principle is based on the method of continuity and constant rank theorems. Caffarelli and Friedman \cite{CaffarelliFriedman1985} developed this approach for semilinear equations in dimension two, and Korevaar and Lewis \cite{KorevaarLewis1987} extended it to higher dimensions. Guan and Ma \cite{GuanMa2003} developed a corresponding constant rank method for fully nonlinear  equations. We refer interested readers to Chen and Ma \cite{ChenMa2018} for more references.

Lions \cite{Lions1981} raised the question whether positive solutions of semilinear Dirichlet problem in convex domains have always convex superlevel sets. Hamel, Nadirashvili and Sire \cite{HamelNadirashviliSire2016} constructed a counterexample in dimension two, hence the quasiconcavity of positive solution may fail for general semilinear Dirichlet problems in convex domains. Zhang \cite{Zhang2026StrictlyStable} recently shows such failure even for minimal, strictly stable, positive solutions on smooth uniformly convex planar domains.

For the Robin problems, much less is known. Let $\Omega\subset\mathbb R^n$, $n\geq2$, be a bounded domain, let $\nu$
	be its outward unit normal, and let $\alpha>0$. We consider the Robin
	torsion problem
	\begin{equation}\label{RobinTorsion}
		\begin{cases}
		-\Delta u^{T,\alpha}=1 & \text{in }\Omega,\\
		\partial_\nu u^{T,\alpha}+\alpha u^{T,\alpha}=0
		& \text{on }\partial\Omega,
		\end{cases}
	\end{equation}
	and the first Robin eigenvalue problem
	\begin{equation}\label{RobinEigen}
		\begin{cases}
		-\Delta u^{E,\alpha}=\lambda^\alpha u^{E,\alpha}
		& \text{in }\Omega,\\
		\partial_\nu u^{E,\alpha}+\alpha u^{E,\alpha}=0
		& \text{on }\partial\Omega.
		\end{cases}
	\end{equation}
	The eigenfunction $u^{E,\alpha}$ is chosen positive and normalized by
	$\|u^{E,\alpha}\|_{L^2(\Omega)}=1$. For the Robin problem and its spectrum, see
	\cite{BucurFreitasKennedy2017,Laugesen2019}.

	In contrast with the Dirichlet problem, a positive first Robin eignefunction does not always preserve log-concavity. 
	In their seminal work, Andrews, Clutterbuck and Hauer
\cite{AndrewsClutterbuckHauer2020} constructed smooth convex domains
for which the first Robin eigenfunction has nonconvex superlevel sets,
and hence is not quasiconcave, when $\alpha>0$ is sufficiently small.
 The argument in the proof of
\cite[Theorem~1.3]{AndrewsClutterbuckHauer2020} also implies that  Robin torsion functions may
have nonconvex superlevel sets for sufficiently small positive
$\alpha$ on bounded convex domains. Meanwhile, they proposed the following 
	conjecture.
	\begin{conjecture}
	\label{conj:ACH}
		For a bounded convex domain $\Omega\subset\mathbb R^n$, there
		exists $\alpha_0=\alpha_0(\Omega)>0$ such that the first Robin
		eigenfunction is log-concave for all $\alpha\geq\alpha_0$.
	\end{conjecture}
   The intuitive argument is that as $\alpha$ goes to infinity, the solutions of the Robin problem converge  to the corresponding Dirichlet solutions. For example, the normalized positive solutions to \eqref{RobinEigen} would converge to the normalized first Dirichlet eigenfunction. Crasta and Fragal\`a proved the conjecture under a boundary regularity assumption whose order increases with the dimension.
	\begin{theorem}[Theorems 1.1-1.2 in \cite{CrastaFragala2021}
	\label{thm:CF-published}]
		Let $\Omega\subset\mathbb R^n$ be bounded, uniformly convex, and
		$\partial\Omega\in C^m$, with
		$\lfloor m-n/2\rfloor\geq4$.
		Then there exists $\alpha_0>0$ such that, for every
		$\alpha\geq\alpha_0$,
		\begin{enumerate}[label=\textup{(\roman*)}]
			\item the first Robin eigenfunction is strictly log-concave;
			\item the Robin torsion function is strictly
			$\frac12$-concave.
		\end{enumerate}
	\end{theorem}
	Here $\lfloor t\rfloor$ means the integer part function. They also tracked the threshold: for fixed $n$ and $m$, $\alpha_0$ can be chosen
	to depend continuously and increasingly on the diameter of the domain $\Omega$, the $C^m$ size of $\partial\Omega$, and decreasingly on the minimum principal curvature of $\partial\Omega$.

	Notice that the regularity assumption for the convex domain depends on the dimension in Theorem \ref{thm:CF-published}, and this dependence was necessary for the approach in \cite{CrastaFragala2021}, since the
	high-order Sobolev estimates and the Morrey--Sobolev embedding were used
	to obtain the $C^2$ convergence; see
	\cite[Sections~2--3]{CrastaFragala2021}. Crasta and Fragal\`a then asked
	the following question \cite[p.~181]{CrastaFragala2021}.

	\begin{problem}
	\label{prob:CF-open}
		Is it possible to remove or weaken the regularity assumptions of Theorems 1.1 and 1.2 in \cite{CrastaFragala2021}?
	\end{problem}

	In this paper,  we solve the above problem in any dimension when the domain  is $C^{3,1}$. Thus our regularity assumption for the domain is dimensional free, and much weaker than that in \cite{CrastaFragala2021}. This also implies Conjecture~\ref{conj:ACH} of Andrews, Clutterbuck and Hauer for uniformly convex $C^{3,1}$ domains. Our main
	results are the following.

	\begin{theorem}[Robin torsion]\label{thm:torsion-concavity}
		Let $\Omega\subset\mathbb R^n$ be a bounded uniformly convex domain of
		class $C^{3,1}$. Then there is $\alpha_{0,T}>0$ such that
		$u^{T,\alpha}$, the Robin torsion function given by \eqref{RobinTorsion}, is strictly $\frac12$-concave for every
		$\alpha\geq\alpha_{0,T}$. The threshold depends only on $n$, the
		$C^{3,1}$ bound for $\partial\Omega$, and the positive lower bound for the
		principal curvatures on $\partial\Omega$.
	\end{theorem}

	\begin{theorem}[First Robin eigenfunction]
	\label{thm:eigen-concavity}
		Let $\Omega\subset\mathbb R^n$ be a bounded uniformly convex domain of
		class $C^{3,1}$. Then there is $\alpha_{0,E}>0$ such that the first
		Robin eigenfunction $u^{E,\alpha}$, solution to \eqref{RobinEigen}, is strictly log-concave for every
		$\alpha\geq\alpha_{0,E}$. The threshold has the same dependence as
		in the Robin torsion problem.
	\end{theorem}
    A main ingredient, which is also of independent interest, is the following global $C^2$-decay estimate. We state them separately.
	For the decay statements, let $u^{T,D}$ solve
	\begin{equation}\label{DirichletTorsion}
		\begin{cases}
			-\Delta u^{T,D}=1 & \text{in }\Omega,\\
			u^{T,D}=0 & \text{on }\partial\Omega,
		\end{cases}
	\end{equation}
	and let $u^{E,D}>0$ be the first Dirichlet eigenfunction satisfying
	$\|u^{E,D}\|_{L^2(\Omega)}=1$,
	\begin{equation}\label{DirichletEigen}
		\begin{cases}
			-\Delta u^{E,D}=\lambda^D u^{E,D} & \text{in }\Omega,\\
			u^{E,D}=0 & \text{on }\partial\Omega.
		\end{cases}
	\end{equation}

	\begin{theorem}[Convergence rate for the torsion problem]
	\label{thm:torsion-decay}
		Let $\Omega\subset\mathbb R^n$ be a bounded domain of class $C^{3,1}$. There are positive constants $C_T$ and $\alpha_T$,
		depending only on the $C^{3,1}$ bound of $\partial\Omega$
		 such that for every $\alpha\geq\alpha_T$,
		\[
		\|u^{T,\alpha}-u^{T,D}\|_{C^2(\overline\Omega)}
		\leq C_T\alpha^{-1}.
		\]
	\end{theorem}

	\begin{theorem}[Convergence rate for the first eigenfunction]
	\label{thm:eigen-decay}
		Let $\Omega\subset\mathbb R^n$ be a bounded domain of class $C^{3,1}$. There are uniform positive constants $C_E$ and $\alpha_E$,
		depending only on the $C^{3,1}$ bound for $\partial\Omega$,
		 such that, for every $\alpha\geq\alpha_E$,
		\[
		\|u^{E,\alpha}-u^{E,D}\|_{C^2(\overline\Omega)}
		\leq C_E\alpha^{-1}.
		\]
	\end{theorem}
For a function $g$ which is twice differentiable up to the boundary, set
	\[
	\|g\|_{C^2(\overline\Omega)}
	:=\|g\|_{L^\infty(\Omega)}
	+\|Dg\|_{L^\infty(\Omega)}
	+\|D^2g\|_{L^\infty(\Omega)}.
	\]
	Hereafter, $D$ and $D^2$ denote the gradient and the hessian of a gereric function. Thus Theorem \ref{thm:torsion-decay} and \ref{thm:eigen-decay} assert respectively the uniform convergence of the functions $u^{T,\alpha}$ and $u^{E,\alpha}$, their
	gradient, and also their second-order derivatives. 

	\subsection{Outline of the proof for convergence}
    Our proof relies on maximum principle and barrier arguments, which is fundamentally different from the method developed by Crasta and Fragal\`a \cite{CrastaFragala2021}.
    
	For the torsion problem, we consider
	\begin{equation}\label{eq:intro-torsion-notation}
	 v^{T,\alpha}:=u^{T,\alpha}-u^{T,D}.
	\end{equation}
	Then $v^{T,\alpha}$ satisfies
	\begin{equation}\label{eq:intro-torsion-problem}
	\left\{
	\begin{aligned}
	 \Delta v^{T,\alpha}&=0 &&\text{in }\Omega,\\
	 \partial_\nu v^{T,\alpha}+\alpha v^{T,\alpha}&=q^T
	 &&\text{on }\partial\Omega,
	\end{aligned}
	\right.
	\end{equation}
	where  $ q^T:=-\partial_\nu u^{T,D}.$

The standard maximum principle first gives directly $\|v^{T,\alpha}\|_{L^\infty(\Omega)}\leq C\alpha^{-1}$. Let $u^1$ be the harmonic extension of $q^T$, and set $w:=u^1-\alpha v^{T,\alpha}$. We remark that $w$ plays a key role in the decay estimates.
     
     Remark that $w$ depends on $\alpha$, but we omit the index to simplify the writing. A boundary maximum argument will give $\|w\|_{L^\infty(\Omega)}\leq C\alpha^{-1}$. Since $w=\partial_\nu v^{T,\alpha}$ on $\partial\Omega$, this also yields the decay estimate for the normal derivative of $v^{T,\alpha}$. We will construct suitable barrier functions and apply the maximum principle to obtain the global gradient estimate for $v^{T,\alpha}$. Similar arguments will provide the uniform gradient estimate for $w$. 

The second derivative estimate is obtained in two steps. First,
similar to the approach of Lions, Trudinger and Urbas
\cite{LionsTrudingerUrbas1986} to the Neumann problem for the
Monge--Amp\`ere equation, we construct auxiliary functions involving
$w$ and show that the global second derivative estimate for
$v^{T,\alpha}$ follows from a bound for its double normal derivative on
$\partial\Omega$. We then obtain this boundary double normal derivative
estimate by constructing a suitable barrier involving $w$ in a boundary
strip, inspired  by Ma and Qiu \cite{MaQiu2019} for 
 the $k$-Hessian Robin problem. The $C^1$-decay estimate
for $w$ is crucial in both steps, and no convexity assumption on
$\Omega$ is required.
	
	We use similar ideas for the eigenvalue problem \eqref{RobinEigen}, where the situation is a little bit more involved, since the error terms are no longer harmonic functions. Consider
	\begin{equation}\label{eq:intro-eigen-notation}
		\begin{aligned}
			v^{E,\alpha}&:=u^{E,\alpha}-u^{E,D},
		\end{aligned}
	\end{equation}
	then  $v^{E,\alpha}$ satisfies
	\begin{equation}\label{eq:intro-eigen-problem}
		\left\{
		\begin{aligned}
			-\Delta v^{E,\alpha}&=\lambda^\alpha v^{E,\alpha}+f^\alpha
			&&\text{in }\Omega,\\
			\partial_\nu v^{E,\alpha}+\alpha v^{E,\alpha}&=q^E
			&&\text{on }\partial\Omega.
		\end{aligned}
		\right.
	\end{equation}
	where $q^E=-\partial_\nu u^{E,D}$ and
	$f^\alpha:=(\lambda^\alpha-\lambda^D)u^{E,D}$.
	
    Notice that as $q^T$, the boundary function $q^E$ are independent of $\alpha$. Since
	$\partial\Omega$ is $C^{3,1}$, standard boundary regularity theory for the
	Dirichlet problems gives $\|q^E\|_{C^{2,\beta}(\partial\Omega)}\leq C$ where $C$ depends only on the $C^{3,1}$ bound for
	$\partial\Omega$, $\beta \in (0, 1)$ and $\lambda^D$.
    
	With the convergence estimate for the eigenvalues  and the $L^2$ decay for $v^{E,\alpha}$ by Filinovskiy  \cite{Filinovskiy2017}, Crasta and Fragal\`a \cite{CrastaFragala2021}, we derive uniform decay estimate
for $v^{E,\alpha}$
    by an argument based on the Alexandrov--Bakelman--Pucci (ABP) estimate
(see Lemma~\ref{lem:ABP-Robin} below). The gradient, and the second-order
	derivative decay estimates follow from similar approach as for the
	torsion problem.

	Once we establish the above decay estimates, the concavity properties follow easily. 
    Therefore, our results hold with the $C^{3,1}$ boundary control.

The paper is organized as follows. In Section~2, we prove the $C^2$-decay estimate for the Robin torsion
function as the Robin parameter tends to infinity. The corresponding
estimate for the first Robin eigenfunction is established in Section~3. Based on these decay estimates, we prove the concavity results in
Section~4.
    
	\section{Convergence for the Robin Torsion Problem}
    \reset
	In this section, we establish the $C^2$ decay estimate for
	$v^{T,\alpha}$. 

\subsection{Notations}
	We introduce first some notations. Unless stated otherwise, $\Omega$ denotes a bounded connected $C^{3,1}$
	domain.  Let $g\in C^{2}(\bar\Omega)$ and denote
	\[
	\partial_\nu g=Dg\cdot\nu,
	\qquad
	g_{\nu\tau}=g_{ij}\nu^i\tau^j,
	\qquad
	g_{\nu\nu}=g_{ij}\nu^i\nu^j,
	\]
	where $\tau, \nu \in \mathbb{R}^n$, $g_i = \partial_{x^i}g$, and $g_{ij} = \partial^2_{x^ix^j}g$. Moreover, repeated indices are summed from $1$ to $n$, by Einstein summation convention.
    
    A $C^k$ or $C^{k,1}$ bound includes the
	diameter, a uniform chart radius, and the corresponding norms of the local
	chart functions. All constants $C$ are positive, generic, and independent of large $\alpha$. 

    We fix also a smooth extension $h$ of the distance function ${\rm dist}(x, \partial\Omega)$ near the boundary. More precisely, $h\in C^{3,1}(\overline\Omega)$ with its norm controlled by the corresponding
	$C^{3,1}$ bound of $\partial\Omega$. For $\mu>0$, define $\Omega_\mu:=\{x\in\Omega:\operatorname{dist}(x,\partial\Omega)<\mu\}$.
    We assume that there is $\mu_*>0$ such that
	\[
	h(x) =\operatorname{dist}(x,\partial\Omega) \quad \mbox{in }\overline\Omega_{\mu_*}.
	\] 
	 Hence $|Dh|=1$ in $\overline\Omega_{\mu_*}$, and 
	$Dh=-\nu$ on $\partial\Omega$. The extension function $h$ can be built from ${\rm dist}(x, \partial\Omega)$ in a tubular neighborhood of $\partial\Omega$ and interior extension; see for
		instance \cite{CrastaMalusa2007}.

	Denote
	\[
	p_T=\min_{\partial\Omega}q^T,
	\qquad
	p_E=\min_{\partial\Omega}q^E,
	\qquad
	g_D=\lambda_2^D-\lambda^D,
	\]
	where $\lambda_2^D$ is the second Dirichlet eigenvalue, $q^T=-\partial_\nu u^{T,D}$, $q^E=-\partial_\nu u^{E,D}$ as defined previously. If $\Omega$ is
	uniformly convex, set
	\[
	\kappa_0=\min_{x\in\partial\Omega}
	\min_{1\leq i\leq n-1}\kappa_i(x)>0,
	\]
    where $\kappa_i(x)$, $1\leq i\leq n-1$, are the principal curvatures of $\partial\Omega$ at $x\in\partial\Omega$.
	The Hopf lemma and simplicity of the first eigenvalue give
	$p_T,p_E,g_D>0$. 
    
	Recall that the
	Dirichlet and Robin torsion functions satisfy respectively 
	\begin{equation}\label{TD1}
		\begin{cases}
		-\Delta u^{T,D}=1 & \text{in }\Omega,\\
		u^{T,D}=0 & \text{on }\partial\Omega,
		\end{cases}\qquad \mbox{and} \qquad \begin{cases}
		-\Delta u^{T,\alpha}=1 & \text{in }\Omega,\\
		\partial_\nu u^{T,\alpha}+\alpha u^{T,\alpha}=0
		& \text{on }\partial\Omega.
		\end{cases}
	\end{equation}

    So $v^{T,\alpha} = v^{T,\alpha} - v^{T,D}$ is harmonic and satisfies \eqref{eq:intro-torsion-problem}, i.e.
	\begin{equation}\label{eq:torsion-difference}
		\begin{cases}
		\Delta v^{T,\alpha}=0 & \text{in }\Omega,\\
		\partial_\nu v^{T,\alpha}+\alpha v^{T,\alpha}=q^T
		& \text{on }\partial\Omega.
		\end{cases}
	\end{equation}

	\subsection[The L-infinity decay estimate]
	{\texorpdfstring{The $L^\infty$ estimate}{The L-infinity estimate}}
Firstly, we give the uniform decay estimate for $v^{T,\alpha}$.
	\begin{lemma}\label{lem:torsion-c0} For any $\alpha>0$, we have \[ \|v^{T,\alpha}\|_{L^\infty(\Omega)} \leq \alpha^{-1}\|q^T\|_{L^\infty(\partial\Omega)} \leq C\alpha^{-1}, \] where $C>0$ depends only on $n$, the $C^2$ bound of $\partial\Omega$, and $\operatorname{diam}(\Omega)$. \end{lemma}
	\begin{proof}
		Since $v^{T,\alpha}$ is harmonic, its maximum and minimum are attained
		on the boundary. Let $x_1$ and $x_2$ be maximum and minimum points,
		respectively. The outer normal derivatives satisfy
		$\partial_\nu v^{T,\alpha}(x_1)\geq0$ and
		$\partial_\nu v^{T,\alpha}(x_2)\leq0$. Then  the boundary condition in
		\eqref{eq:torsion-difference}  gives
		\begin{align*}
		\alpha v^{T,\alpha}(x_1)\leq q^T(x_1),\qquad 
		\alpha v^{T,\alpha}(x_2)\geq q^T(x_2).
		\end{align*}
	So we are done.\end{proof}

	\subsection{The gradient estimate}

    By the $L^{\infty}$ decay estimate, we only have the boundedness of $\partial_{\nu}v^{T_{\alpha}}$ on $\partial\Omega$. To obtain the decay estimate of $|Dv^{T,\alpha}|$,
	we consider the asymptotic behavior of $\alpha v^{T,\alpha}$ as $\alpha$ goes to $\infty$. It is not difficult to imagine its limit is $u^1$,  the harmonic extension of $q^T$, i.e. $u^{1}$ solves
	\begin{equation}\label{harmonic-extension}
		\begin{cases}
		\Delta u^1=0 & \text{in }\Omega,\\
		u^1=q^T & \text{on }\partial\Omega.
		\end{cases}
	\end{equation}
	The $C^{2, \beta}$ bound for $q^T$ and standard elliptic
	estimates give
	\[
	\|Du^1\|_{L^\infty(\Omega)}\leq C,
	\]
	where $C$ depends only on $\Omega$. Denote the error term by
	\begin{equation}\label{wtorsion}
		w:=u^1-\alpha v^{T,\alpha}.
	\end{equation}
 On the boundary $\partial\Omega$,
	\[
	w=q^T-\alpha v^{T,\alpha}=\partial_\nu v^{T,\alpha}.
	\]
	Thus, $w$ solves
	\begin{equation}\label{eq:w-torsion}
		\begin{cases}
		\Delta w=0 & \text{in }\Omega,\\
		\partial_\nu w+\alpha w=\partial_\nu u^1
		& \text{on }\partial\Omega.
		\end{cases}
	\end{equation}
	The next lemma gives the $L^\infty$ decay estimate for $w$, and hence the
	normal derivative estimate for $v^{T,\alpha}$.
	\begin{lemma}\label{lem:w-c0} For any $\alpha>0$, we have \begin{align}\label{westimate1}
    \|w\|_{L^\infty(\Omega)} \leq \alpha^{-1}\|\partial_\nu u^1\|_{L^\infty(\partial\Omega)}\le C\alpha^{-1}. \end{align} 
    where $C$ depends only on $n$, the $C^3$ bound of $\partial\Omega$, and $\operatorname{diam}(\Omega)$. In particular, \[\|\partial_\nu v^{T,\alpha}\|_{L^\infty(\partial\Omega)} \leq C\alpha^{-1}. \] \end{lemma}

	\begin{proof}
		Since $w$ is harmonic, it attains its maximum and minimum on
		$\partial\Omega$. Let $x_1$ be a maximum point, then
		$\partial_\nu w(x_1)\geq0$. Hence the boundary condition in
		\eqref{eq:w-torsion} gives
		\begin{align*}
		\alpha w(x_1)\leq \partial_\nu u^1(x_1)
		\leq \|\partial_\nu u^1\|_{L^\infty(\partial\Omega)}.
		\end{align*}
		Let $x_2$ be a minimum point, then
		$\partial_\nu w(x_2)\leq0$, and therefore
		\[
		\alpha w(x_2)\geq \partial_\nu u^1(x_2)
		\geq-\|\partial_\nu u^1\|_{L^\infty(\partial\Omega)}.
		\]
		This gives \eqref{westimate1}. The estimate for
		$\partial_\nu v^{T,\alpha}$ follows because
		$w=\partial_\nu v^{T,\alpha}$ on $\partial\Omega$.
	\end{proof}
	
	Next, we prove the decay estimates for $Dv^{T,\alpha}$ and $Dw$.
	\begin{lemma}[Gradient estimates]\phantomsection
		\label{lem:torsion-c1}
        There are positive constants $C_{1}$ depending only on the $C^3$ bound for
			$\partial\Omega$, and $C_{2}$ depending only on the $C^{3,1}$ bound for
			$\partial\Omega$, such that, for large enough $\alpha$,
			\[
			\|Dv^{T,\alpha}\|_{L^\infty(\Omega)}
			\leq C_{1}\alpha^{-1}\quad \mbox{and} \quad
			\|Dw\|_{L^\infty(\Omega)}
			\leq C_{2}\alpha^{-1}.
			\]
	\end{lemma}

	\begin{proof}
		Let $h$ be the extension of distance function fixed above. Let
		\[
		G(x,\xi)= {\langle \xi, Dv^{T,\alpha}\rangle}
		{ + } \langle\xi,Dh\rangle w+\frac{w^2}{2}, \quad (x, \xi) \in \overline{\Omega}\times \mathbb S^{n-1}.
		\]
		Suppose that $G$ attains its maximum at $(x_0,\xi_0)$. We claim 
        \begin{align}
        \label{claimG}
            G(x_0,\xi_0) = \max_{\overline\Omega\times \mathbb S^{n-1}} G(x, \xi) \le C\alpha^{-1} \quad \mbox{if $\alpha$ is large enough}.
        \end{align}
        
        Let first $x_0\in\Omega$. Since $v^{T,\alpha}$ and $w$ are harmonic, we have
		\[
		\Delta G(x, \xi_0)
		= \langle\xi_0,D\Delta h\rangle w
		+ 2\langle\xi_0,D^2h Dw\rangle+|Dw|^2.
		\]
		Without specific mention, the derivatives are always for the variable $x$, and $|\cdot|$ stands for the Euclidean norm. At $x_0$, using $|\xi_0| - 1$, there holds
		\[
		0\geq\Delta G(x_0,\xi_0)
		\geq |Dw|^2-C|Dw|-C.
		\]
		Thus $|Dw(x_0)|\leq C$. Since
		$Dv^{T,\alpha}=\alpha^{-1}(Du^1-Dw)$, it follows that
		$G(x_0,\xi_0)\leq C\alpha^{-1}$ using \eqref{westimate1}.

		Assume next that $x_0\in\partial\Omega$. 
        
        {\bf Subcase 1.} Assume that $\xi_0$ is tangent to
		$\partial\Omega$ at $x_0$. We choose the coordinates such that $x_0=0$,
		$\nu(0)=-e_n$, $\xi_0=e_1$, and 
$\partial\Omega$ is represented locally by $x_n=\rho(x')$ with $\rho(0')=0$, $\nabla \rho(0')=0'$.

Notice that given any $2\le k \le n-1$, as $0 \ge G(x_0, \pm e_k) - G(x_0, e_1) = \pm v_k^{T, \alpha} - v_1^{T, \alpha}$ so that $v_1^{T, \alpha}(x_0) \ge |v_k^{T, \alpha}|(x_0)$ for $2\le k \le n-1$. On the other hand, as $\nu = (\nu^k)_{1\le k\le n}$ and $|\nu|^2 \equiv 1$, we get then $-\partial_1\nu^n(x_0) = \sum_{1 \le k \le n} \nu^k\partial_1\nu^k(x_0) = 0$.

		Differentiating
		$\partial_\nu v^{T,\alpha}=w$ in $e_1$-direction at $x_0$ gives
		\begin{align*}
		    -v^{T,\alpha}_{n1}
		= u^1_1- \alpha
		v^{T,\alpha}_1 - \sum_{k=1}^n \partial_1\nu^k v_k^{T, \alpha} & = u^1_1- \alpha
		v^{T,\alpha}_1 - \sum_{k=1}^{n-1} \partial_1\nu^k v_k^{T, \alpha}\\
        & \leq u^1_1 - \alpha v^{T,\alpha}_1 + Cv^{T,\alpha}_1.
		\end{align*}
Moreover,
		\[
		\partial_\nu G(x_0,e_1)
		=v^{T,\alpha}_{1\nu}
		+ \langle e_1,D^2h\,\nu\rangle w+w\,\partial_\nu w\geq0.
		\]
		 where $\partial_\nu w=\partial_\nu u^1-\alpha w$ on the boundary is uniformly bounded for $\alpha$ large by Lemma \ref{lem:w-c0}. Combining
		the preceding estimate for $v^{T,\alpha}_{1\nu}(x_0) = -v^{T,\alpha}_{n1}(x_0)$, the
		bounds for $w$, $u^1$ and $Dh$, we arrive at
		\[
		(\alpha-C)v^{T,\alpha}_1(x_0)\leq C.
		\]
		This gives $G(x_0,e_1)\leq C\alpha^{-1}$ for
		$\alpha$ large enough.

		{\bf Subcase 2.} Assume that $\xi_0$ is not tangent to $\partial\Omega$ at $x_0$. We write
		$\xi_0=a\tau+b\nu(x_0)$, where $\tau$ is tangent, unitary, and
		$a=\sqrt{1-b^2}<1$. Since $\partial_\nu v^{T,\alpha}=w$ on $\partial\Omega$,
		\[
		G(x_0,\xi_0)
		=aG(x_0,\tau)+ \frac{1-a}{2}w(x_0)^2 \le aG(x_0,\xi_0)+ \frac{1-a}{2}w(x_0)^2.
		\]
		Lemma \ref{lem:w-c0} yields then
		$G(x_0,\xi_0)\leq C\alpha^{-2}$ for large $\alpha$. 
        
        Therefore we get the claim \eqref{claimG}. Combining with the uniform estimate for $w$ and the harmonicity of $Dv^{T,\alpha}$, we obtain
		$$\|Dv^{T,\alpha}\|_{L^\infty(\Omega)} = \|Dv^{T,\alpha}\|_{L^\infty(\partial\Omega)} \leq C\alpha^{-1}.$$
		Remark that the boundary calculation uses only $D\nu$ and hence a $C^2$ bound for
		$\partial\Omega$, the interior calculation contains $D\Delta h$, which
		accounts for the $C^3$ bound of $\partial\Omega$.

		\smallskip
		For the gradient estimate of $w$, let $\widetilde u$ be solution to
	\[
	\begin{cases}
		\Delta \widetilde u=0 & \text{in }\Omega,\\
		\widetilde u=\partial_\nu u^1 & \text{on }\partial\Omega,
	\end{cases}
	\]
	and set $\widetilde w:=\widetilde u-\alpha w$. Since
	$\partial_\nu w+\alpha w=\partial_\nu u^1=\widetilde u$ on
	$\partial\Omega$, there holds
	$\widetilde w=\partial_\nu w$ on $\partial\Omega$. Hence
	$\widetilde w$ solves
	\[
	\begin{cases}
		\Delta \widetilde w=0 & \text{in }\Omega,\\
		\partial_\nu\widetilde w+\alpha\widetilde w
		=\partial_\nu\widetilde u & \text{on }\partial\Omega.
	\end{cases}
	\]
		Using the preceding $C^{2,\beta}$ bound for $q^T$, and the elliptic
regularity theory, applied successively to the Dirichlet problems for $u^1$
		and $\widetilde u$, we can claim
		\[
		\|D^2u^1\|_{L^\infty(\Omega)}
		+\|D\widetilde u\|_{L^\infty(\Omega)}\leq C.
		\]
		By  maximum principle, $\widetilde w$ satisfies
		\[
		\|\widetilde w\|_{L^\infty(\Omega)}
		\leq\alpha^{-1}
		\|\partial_\nu\widetilde u\|_{L^\infty(\partial\Omega)}
		\leq C\alpha^{-1}.
		\]
		Set now
		\[
		\widetilde G(x,\xi)
		= \langle \xi, Dw\rangle {\color{red} + }\langle\xi,Dh\rangle\widetilde w
		+\frac{\widetilde w^2}{2}.
		\]
        Using $\widetilde G$, $\widetilde u$, $w$, and $\widetilde w$ in place of $G$, $u^1$, $v^{T, \alpha}$, and $w$, the repeat of the previous estimates for $\|Dv^{T,\alpha}\|_{L^\infty(\Omega)}$ yields $\|Dw\|_{L^\infty(\Omega)}
			\leq C\alpha^{-1}$ for $\alpha$ large. So we omit the details.
	\end{proof}
	
	\subsection{The second order derivative estimate}
	
	We first reduce the second-order derivative estimate to the bound of double normal derivatives on
	the boundary. Remark that the convexity of the domain is not needed for this step.
	\begin{lemma}\label{lem:torsion-hessian-reduction}
		There is $C>0$,
		depending only on the $C^{3,1}$ bound of
		$\partial\Omega$, such that
		\begin{align}
			\|D^2v^{T,\alpha}\|_{L^\infty(\Omega)}
			\leq C\left(\alpha^{-1}
			+\|v^{T,\alpha}_{\nu\nu}\|_{L^\infty(\partial\Omega)}\right)
			\end{align}
		for all $\alpha$ large enough.
	\end{lemma}
	\begin{proof}
		Consider the following auxiliary function
		\begin{align}
			W(x,\xi)=v^{T,\alpha}_{\xi\xi}(x) {+ } 2\langle \xi, {Dh}\rangle \langle \widetilde \xi, Dw {+ }D^2h D v^{T,\alpha}\rangle +\frac{|Dw|^2}{2},
		\end{align}
		where $\widetilde \xi=\xi-\langle\xi,Dh \rangle Dh$.
		Assume that $W$ attains its maximum over $\overline\Omega\times {\mathbb S}^{n-1}$ at $(x_0, \xi_0)$. We may assume $v^{T,\alpha}_{\xi_0\xi_0}(x_0)\ge {\alpha}^{-1}$,
		otherwise $W(x_0,\xi_0)\le C\alpha^{-1}$.
		
		From $W(x_0,e_i)\le W(x_0,\xi_0)$, we obtain $v^{T,\alpha}_{ii}(x_0)\le v^{T,\alpha}_{\xi_0\xi_0}+C\alpha^{-1}$. Thus
		$$v^{T,\alpha}_{ii}=-\sum_{j\neq i} v^{T,\alpha}_{jj}\ge -(n-1)v^{T,\alpha}_{\xi_0\xi_0}-C\alpha^{-1}.$$
		
		Therefore
		\begin{align*}
			\max_{1\le i\le n} |v^{T,\alpha}_{ii}|\le C_n v^{T,\alpha}_{\xi_0\xi_0}+C\alpha^{-1}.
		\end{align*}
		For $i\neq j$, from
		\begin{align*}
			2v^{T,\alpha}_{ij}
			=2v^{T,\alpha}_{\frac{e_i+e_j}{\sqrt{2}},\,
			\frac{e_i+e_j}{\sqrt{2}}}
			-v^{T,\alpha}_{ii}-v^{T,\alpha}_{jj}
		\end{align*}
		we obtain $|v^{T,\alpha}_{ij}|\le 2n v^{T,\alpha}_{\xi_0\xi_0}+C\alpha^{-1}$. Hence
		\begin{align}
        \label{estimate-c2-vTa}
			|D^2 v^{T,\alpha}(x_0)|\le C_n v^{T,\alpha}_{\xi_0\xi_0}(x_0)+C\alpha^{-1}.
		\end{align}
		
		\textbf{Case 1: $x_0\in \Omega$.}
		Recalling that $w$ and $v^{T, \alpha}$ are harmonic and $w = u^1 - \alpha v^{T, \alpha}$, direct differentiation at $x_0$ gives
		\begin{align*}
			0\ge \Delta W
			& \ge |D^2 w|^2
			- C\bigl(|D^2w|+|D^2v^{T,\alpha}|
			+|Dw|+|Dv^{T,\alpha}|\bigr)\\
			& \ge -C+\frac{{\alpha}^2}{4}|D^2v^{T,\alpha}|^2.
		\end{align*}
		We get $|D^2v^{T,\alpha}(x_0)| \le C\alpha^{-1}$ for large $\alpha$, hence 
		\begin{align}
        \label{estimate-W}
			W(x,\xi)\le W(x_0,\xi_0)\le C\alpha^{-1}, \quad \forall\; (x,\xi) \in \overline\Omega\times {\mathbb S}^{n-1}.
		\end{align}
        
        By the definition of $W$, $D^2W$ involves at most fourth order derivatives
		of $h$, the extension of distance function. This is the only point where the Lipschitz bound for $D^3h$ is
		used. We give the passage to $C^{3,1}$ bound more properly. Choose a ball
		$B\Subset\Omega$ centred at $x_0$ and mollify $h$ inside $B$. The
		mollifications $h_\varepsilon$ satisfy
		\[
		\|D^j h_\varepsilon\|_{L^\infty(B)}\leq C\|h\|_{C^{3,1}}
		\quad \forall\; j \le 4,
		\]
		with constants independent of $\varepsilon > 0$ sufficiently small. 
		We consider $$W_{\varepsilon}=v^{T,\alpha}_{\xi\xi}(x) { + } 2\langle {\xi}, {Dh_{\varepsilon}}\rangle \langle \widetilde \xi, Dw { + } D^2h_{\varepsilon} Dv^{T,\alpha}\rangle +\frac{|Dw|^2}{2}-\delta\big(|x-x_0|^2+|\xi-\xi_0|^2\big).$$  
        The corresponding local
		maxima for $W_\varepsilon$ converge to
		$(x_0,\xi_0)$. The preceding calculation is classical at those local maxima points, and
		provides estimate independent of $\varepsilon$. Letting first
		$\varepsilon\to0$ and then $\delta\to0$ proves the same inequality for
		$h\in C^{3,1}$. In other words, the constant $C$ in \eqref{estimate-W} depends only on the $C^{3,1}$ regularity of the domain.
        
	\textbf{Case 2: $x_0\in\partial \Omega$ and $\xi_0$ is non-tangential at $x_0$.}
Here we write $\xi_0=a\tau+b\nu$, where $\tau$ is a unit tangent vector at $x_0$, and
	$a=\sqrt{1-b^2} < 1$. 
	
	Differentiating the boundary condition $\partial_\nu v^{T,\alpha}=w$ in a tangential direction, Lemma \ref{lem:torsion-c1} implies that
	\begin{align}\label{vnutau}
		v^{T,\alpha}_{\nu\tau} = - \sum_{k=1}^n v^{T,\alpha}_k\nu^k_{\tau} + w_\tau = O(\alpha^{-1}).
	\end{align}
	Using \eqref{vnutau} and  the definition of $W$, we have
	\begin{align*}
	W(x_0,\xi_0) & = a^2v^{T,\alpha}_{\tau\tau}+b^2v^{T,\alpha}_{\nu\nu}  { + 2ab v^{T,\alpha}_{\nu\tau} -2ab\langle \tau, Dw+D^2h D v^{T,\alpha}\rangle} + \frac{|Dw|^2}{2}\\
    & = a^2W(x_0,\tau) + b^2v^{T,\alpha}_{\nu\nu}+ \frac{b^2}{2}|Dw|^2 \\
	&\le a^2W(x_0,\xi_0)+b^2v^{T,\alpha}_{\nu\nu}+\frac{b^2}{2}|Dw|^2 \\
    & = a^2W(x_0,\xi_0)+b^2W(x_0,\nu) .
	\end{align*}
	Thus
	\begin{align*}    
		W(x_0,\xi_0)\le W(x_0,\nu)  \le C\big({\alpha}^{-1}+\|v^{T,\alpha}_{\nu\nu}\|_{L^\infty(\partial\Omega)}\big),
	\end{align*}
	where we used the decay estimate for $|Dw|$.
	
	\textbf{Case 3: $x_0 \in \partial\Omega$ and $\xi_0$ is tangential at $x_0$.}
	As before, we can assume that $x_0=0$, $\xi_0=e_1$, and $\nu(x_0)=-e_n$. Write the
	boundary near $x_0$ as $x_n=\rho(x')$, where
	$\rho(0')=0$ and $D\rho(0')=0'$. Differentiating
	$\partial_\nu v^{T,\alpha} =w$ twice along $e_1$, and using
	$w=u^1-\alpha v^{T,\alpha}$, there holds
	\[
	v^{T,\alpha}_{n11}
	=\alpha v^{T,\alpha}_{11}
	+O\bigl(1+|Dv^{T,\alpha}|+ |D^2v^{T,\alpha}|\bigr) = \alpha v^{T,\alpha}_{11}
	+O\bigl(1 + |D^2v^{T,\alpha}|\bigr).
	\]
	Moreover, since $e_n$ is the inward normal and $W(\cdot,e_1)$ reaches the maximum at
	$x_0$, $W_n(x_0,e_1)\leq0$. Direct differentiation gives
	\[
		W_n(x_0,e_1)=v^{T,\alpha}_{11n}+ \sum_{k=1}^n w_kw_{kn} + O\bigl(|Dv^{T,\alpha}|+|Dw|\bigr).
		\]
	Using Lemma~\ref{lem:torsion-c1}, the identity
	$D^2w=D^2u^1-\alpha D^2v^{T,\alpha}$, and \eqref{estimate-c2-vTa}, we obtain
	\[
	0\geq W_n(x_0,e_1)
	\geq(\alpha-C)v^{T,\alpha}_{11}-C.
	\]
	We deduce then $v^{T,\alpha}_{11}\leq C\alpha^{-1}$ for large $\alpha$, hence 
	\[
	W(x_0,\xi_0)\leq C\Big(\alpha^{-1}
	+\|v^{T,\alpha}_{\nu\nu}\|_{L^\infty(\partial\Omega)}\Big).
	\]
    The proof is completed.
	\end{proof}
    
	The next step is to control the double normal derivatives.   
	  
\begin{lemma}[Double normal derivative estimate] 
\label{lem:torsion-double-normal}  
	There is a uniform constant $C>0$,  
	depending only on $C^{3,1}$ bound of  
	$\partial\Omega$, such that  
	\[  
\|v^{T,\alpha}_{\nu\nu}\|_{L^\infty(\partial\Omega)}  
	\leq C\alpha^{-1}  
	\]  
	holds for $\alpha$ sufficiently large.  
\end{lemma}  
  
\begin{proof}  
	Set 
	$
M:=\|v^{T,\alpha}_{\nu\nu}\|_{L^\infty(\partial\Omega)}$   
	{and} 
	$
	\varphi:=\langle Dv^{T,\alpha},Dh\rangle+w.  
	$
	Since $Dh=-\nu$ and $w=\partial_\nu v^{T,\alpha}$ on  
	$\partial\Omega$, we have $\varphi=0$ on  
	$\partial\Omega$. Recall that  
	$v^{T,\alpha}$ and $w$ are harmonic,  we have
	\[  
	\Delta\varphi  
	=2\langle D^2v^{T,\alpha},D^2h\rangle  
	+\langle Dv^{T,\alpha},D(\Delta h)\rangle.  
	\]  
	The preceding bound for $w$, together with  
	Lemmas~\ref{lem:torsion-c1} and  
	\ref{lem:torsion-hessian-reduction}, gives us positive constants  
	$C_1$ and $\alpha_1$ such that for any $\alpha\geq\alpha_1$  
	\begin{equation}\label{eq:phi-distance-bounds}  
	|\varphi|\leq C_1\alpha^{-1},  
	\quad  
	|\Delta\varphi|\leq C_1\big(M+\alpha^{-1}\big)  
	\quad\text{in }\Omega_{\mu_*}  
	\end{equation}  
	and  
	\begin{equation}\label{eq:torsion-distance-boundary-error}  
	|\partial_\nu w|\leq C_1\alpha^{-1}  
	\quad\text{on }\partial\Omega.  
	\end{equation}  
  
	Choose  
	\[  
	C_0\geq  
	\|\Delta h\|_{L^\infty(\Omega_{\mu_*})}+8C_1+1,  
	\quad  
	\mu_0:=\min\left\{\mu_*,\frac{1}{2C_0}\right\} 
	<\frac{1}{2}.  
	\]  
	We claim that  
	\[  
	M\leq\frac{16C_1}{3\mu_0}\alpha^{-1}, 
	\quad \forall\;\alpha\geq\alpha_1.  
	\]  
	Suppose the contrary, that is,   
	\begin{equation}\label{eq:M-large-distance}  
	\alpha\geq\alpha_1\;\mbox{ and }\; 
	M>\frac{16C_1}{3\mu_0}\alpha^{-1}.  
	\end{equation}  
	Take  
	\[  
	g:=h-\frac{C_0}{2}h^2  
	\quad\text{in }\Omega_{\mu_0}.  
	\]  
	In $\Omega_{\mu_0}$, there hold  
	\begin{equation}\label{eq:h-distance-properties}  
	g\geq\frac34h,  
	\quad  
	\Delta g=(1-C_0h)\Delta h-C_0\leq-8C_1-1.  
	\end{equation}  
	Moreover,  
	\begin{equation}\label{eq:h-distance-boundary}  
	g=0,\quad \partial_\nu g=-1  
	\quad\text{on }\partial\Omega.  
	\end{equation}  
  
	By definition of $M$, one of the following two cases occurs:  
	there is $x_+$ or $x_-\in\partial\Omega$ such that  
	$M=v^{T,\alpha}_{\nu\nu}(x_+)$ or  
	$M=-v^{T,\alpha}_{\nu\nu}(x_-)$.  
  
	\textbf{Case 1: $x_+\in\partial\Omega$ and  
	$M=v^{T,\alpha}_{\nu\nu}(x_+)$.} 
    
    \smallskip
\noindent
	Let $\alpha$ satisfy \eqref{eq:M-large-distance}, consider the function  
	\[  
	P:=\varphi-\frac{M}{4}g.  
	\]  
	So $P=0$ on $\partial\Omega$. On the inner boundary of  
	$\Omega_{\mu_0}$, i.e. $\{h=\mu_0\}$,  
	\eqref{eq:phi-distance-bounds},  
	\eqref{eq:h-distance-properties}, and  
	\eqref{eq:M-large-distance} give  
	\[  
	P\leq C_1\alpha^{-1}-\frac{3\mu_0}{16}M<0.  
	\]  
	Let $\max_{\overline\Omega_{\mu_0}}P$ be attained at $x_1$.  
	If $P(x_1)>0$, then $x_1\in\Omega_{\mu_0}$, so that  
	\begin{align*}  
	0\geq\Delta P(x_1)  
	&\geq-C_1\left(M+\alpha^{-1}\right)  
	+\frac{M}{4}(8C_1+1)\\  
	&=\Big(C_1+\frac14\Big)M-C_1\alpha^{-1}\\  
	&>\Big(C_1+\frac14-\frac{3\mu_0}{16}\Big)M>0,  
	\end{align*}  
	where 
	where we used \eqref{eq:M-large-distance} and  
	$\mu_0<\frac12$.  
	This is impossible, hence  
	$P\leq0$ in $\overline\Omega_{\mu_0}$.  
  
	Furthermore, since  
	$h(x)=\operatorname{dist}(x,\partial\Omega)$ in  
	$\Omega_{\mu_*}$, differentiating $|Dh|^2=1$ yields 
	\[ 
	D^2hDh=0\quad\text{in }\Omega_{\mu_*}. 
	\] 
	Since $Dh=-\nu$ on $\partial\Omega$, it follows that 
	\begin{equation}\label{eq:distance-normal-hessian} 
	D^2h\,\nu=0\quad\text{on }\partial\Omega. 
	\end{equation} 
	Therefore, on $\partial\Omega$, 
	\[ 
	\partial_\nu\varphi 
	=-v^{T,\alpha}_{\nu\nu}+\partial_\nu w. 
	\] 
	Since $P\leq0$ in $\Omega_{\mu_0}$ and  
	$P=0$ on $\partial\Omega$,  
	$\partial_\nu P(x_+)\geq0$. By  
	\eqref{eq:h-distance-boundary},  
	\[  
	0\leq\partial_\nu P(x_+)  
	\leq-\frac34M+C_1\alpha^{-1},  
	\]  
	consequently  
	\[  
	M\leq\frac{4C_1}{3}\alpha^{-1}.  
	\]  
	This contradicts \eqref{eq:M-large-distance}, since  
	$\mu_0<\frac12$.  
  
	\textbf{Case 2: $x_-\in\partial\Omega$ and  
	$M=-v^{T,\alpha}_{\nu\nu}(x_-)$.}  

    \smallskip
\noindent
	The proof is similar, and we will go through quickly.  
	Let $\alpha$ satisfy \eqref{eq:M-large-distance} and  
	\[  
	\widetilde P:=-\varphi-\frac{M}{4}g.  
	\]  
	As above, $\widetilde P=0$ on $\partial\Omega$ and  
	$\widetilde P<0$ on $\{h=\mu_0\}$. If  
	$\max_{\overline{\Omega}_{\mu_0}}\widetilde P$ is positive  
	and is reached in $\Omega_{\mu_0}$, then 
	\begin{align*}  
	0\geq\Delta\widetilde P 
	&\geq-C_1\left(M+\alpha^{-1}\right) 
	+\frac{M}{4}(8C_1+1)\\ 
	&=\Big(C_1+\frac14\Big)M-C_1\alpha^{-1}\\ 
	&>\Big(C+\frac14-\frac{3\mu_0}{16}\Big)M>0, 
	\end{align*} 
	a contradiction. Hence  
	$\widetilde P\leq0$ in $\overline\Omega_{\mu_0}$.  
	Using  
	\[ 
	\partial_\nu\varphi 
	=-v^{T,\alpha}_{\nu\nu}+\partial_\nu w 
	\quad\text{on }\partial\Omega, 
	\] 
	we obtain 
	\[ 
	0\leq\partial_\nu\widetilde P(x_-) 
	\leq-\frac34M+C_1\alpha^{-1}. 
	\] 
	Consequently, 
	\[ 
	M\leq\frac{4C_1}{3}\alpha^{-1}, 
	\] 
	which again contradicts \eqref{eq:M-large-distance}. Therefore, 
	the desired estimate holds with 
	\[ 
	C_{T}=\frac{16C_1}{3\mu_0} 
	\] 
	for every $\alpha\geq\alpha_1$. 
\end{proof}
	Combining Lemmas \ref{lem:torsion-c0}-\ref{lem:torsion-double-normal}, we conclude that for large enough $\alpha$,
	$\|v^{T,\alpha}\|_{C^2(\overline\Omega)}\leq C\alpha^{-1}.$
    Hence Theorem \ref{thm:torsion-decay} holds true. \qed
    
	\section{Convergence for the Robin Eigenfunction Problem}
\reset
In this section, we prove the $C^2$ decay estimate for $v^{E,\alpha}$. 
	Recall that 
    $$v^{E,\alpha}=u^{E,\alpha}-u^{E,D}, \quad f^\alpha:=(\lambda^\alpha-\lambda^D)u^{E,D},\quad \mbox{and}\quad q^E=-\partial_\nu u^{E,D}$$ from
	\eqref{eq:intro-eigen-notation}. Then $v^{E,\alpha}$ satisfies
	\begin{equation}\label{eq:eigen-difference}
		\begin{cases}
		-\Delta v^{E,\alpha}
		=\lambda^\alpha v^{E,\alpha}+f^\alpha
		& \text{in }\Omega,\\
		\partial_\nu v^{E,\alpha}+\alpha v^{E,\alpha}=q^E
		& \text{on }\partial\Omega.
		\end{cases}
	\end{equation}
	The Hopf lemma gives $q^E>0$ on $\partial\Omega$.

	For all sufficiently large $\alpha$, it is known that (see \cite{Filinovskiy2017, CrastaFragala2021}).
	\begin{equation}\label{eq:eigenvalue-decay}
		0\leq\lambda^D-\lambda^\alpha\leq C\alpha^{-1}
	\end{equation}
	Here $C$ depends only on  $C^2$ bound for $\partial\Omega$. Moreover we have the following $L^2$ decay estimate 
	\cite{CrastaFragala2021, Filinovskiy2017}. Here we give an alternative proof.

\begin{lemma}[$L^2$ decay estimate]\label{lem:eigen-l2}
Let $g_D = \lambda_2^D - \lambda^D$ be the spectral gap under the Dirichlet boundary condition. There exist $C_{E,0}>0$, depending only $C^{1,1}$ bound of $\partial\Omega$, $\lambda^D$ and $g_D^{-1}$, such that for $\alpha$ large,
\[
 \|v^{E,\alpha}\|_{L^2(\Omega)}
 \leq C_{E,0}\alpha^{-1}.
\]
\end{lemma}
\begin{proof}
Set
\[
 a_\alpha:=\int_{\Omega} u^{E,\alpha}u^{E,D}\,dx,
 \qquad
 z^\alpha:=u^{E,\alpha}-a_\alpha u^{E,D}.
\]
Since $u^{E,D}$, $u^{E,\alpha}$ are positive,  $\|u^{E,D}\|_{L^2(\Omega)}=\|u^{E,\alpha}\|_{L^2(\Omega)}=1$, we have
\begin{equation}\label{eq:orthogonal-identities}
0<a_\alpha\leq1,
 \quad
\int_\Omega u^{E,D}z^\alpha dx =0, \quad \|z^\alpha\|_{L^2(\Omega)}^2=1-a_\alpha^2=
 \int_\Omega u^{E,\alpha}z^\alpha\,dx.
\end{equation}
Then 
\begin{align*}
\|v^{E,\alpha}\|_{L^2(\Omega)}^2 =2(1-a_\alpha)
 =\frac{2}{1+a_\alpha}\|z^\alpha\|_{L^2(\Omega)}^2.
\end{align*}
To prove the lemma, we only need to prove $\|z^{\alpha}\|_{L^2{(\Omega)}}\le C\alpha^{-1}$.

For $\alpha \ge 1$, since
\[
 \|u^{E,\alpha}\|_{H^1(\Omega)}^2 \le \|Du^{E,\alpha}\|_{L^2(\Omega)}^2
 +\alpha\|u^{E,\alpha}\|_{L^2(\partial\Omega)}^2
 =\lambda^\alpha\leq\lambda^D,
\]
we have
\begin{equation}\label{eq:normal-bound}
 \|\partial_\nu u^{E,\alpha}\|_{H^{-\frac12}(\partial\Omega)}
 \leq C, \quad \forall \; \alpha \ge 1.
\end{equation}
where $C$ depends only on the $C^{1,1}$ regularity of $\Omega$.

On the other hand, as $z^\alpha\perp u^{E,D}$, there exists a unique
$
 \psi^\alpha\in H^2(\Omega)\cap H_0^1(\Omega)
 \cap\{u^{E,D}\}^{\perp}
$
satisfying
\begin{equation}\label{eq:dual}
 (-\Delta  -\lambda^\alpha)\psi^\alpha=z^\alpha.
\end{equation}
Since
$\lambda_2^D-\lambda^\alpha
 \geq\lambda_2^D-\lambda^D=g_D,$
the spectral theorem gives $\|\psi^\alpha\|_{L^2(\Omega)}
 \leq g_D^{-1}\|z^\alpha\|_{L^2(\Omega)}$.
By the $H^2$ estimate for the Poisson problem on a $C^{1,1}$ domain,
\begin{equation}\label{eq:dual-h2}
 \|\psi^\alpha\|_{H^2(\Omega)}
 \leq C\bigl(1+\lambda^D g_D^{-1}\bigr)
 \|z^\alpha\|_{L^2(\Omega)}.
\end{equation}

Using \eqref{eq:orthogonal-identities}, \eqref{eq:dual},  we obtain
\begin{align*}
 \|z^\alpha\|_{L^2(\Omega)}^2 
=\int_\Omega u^{E,\alpha}
       (-\Delta-\lambda^\alpha)\psi^\alpha dx & = -\int_{\partial\Omega} u^{E,\alpha}\,\partial_\nu\psi^\alpha d\sigma =\alpha^{-1}\int_{\partial\Omega} \partial_\nu u^{E,\alpha}\,\partial_\nu\psi^\alpha,
\end{align*}
Combining the above equality with \eqref{eq:normal-bound}, \eqref{eq:dual-h2} and $\|\partial_\nu\psi^\alpha\|_{H^{\frac12}(\partial\Omega)} \leq C\|\psi^\alpha\|_{H^2(\Omega)}$, we get $\|z^\alpha\|_{L^2(\Omega)}\leq C\alpha^{-1}$ for $\alpha \ge 1$. \end{proof}

	\subsection[The L-infinity estimate]
	{\texorpdfstring{The $L^\infty$ estimate}{The L-infinity estimate}}
	Based on the $L^2$ estimate, we now prove the $L^\infty$ decay estimate for $v^{E, \alpha}$. We first record the following ABP estimate for the Robin problem, which is a special case of \cite[Chapter~1, Section~1.5, Theorem~1.9]{Lieberman2013}. For completeness, we include a proof.
    \begin{lemma}
		\label{lem:ABP-Robin}
		Let $\Omega\subset\mathbb{R}^{n}$ be a bounded $C^{1}$ domain, and 
		$\alpha>0$.
		Suppose that
		$
		u\in C^{2}(\Omega)\cap C^{1}(\overline{\Omega})
		$
		satisfies
		\begin{equation}
			\label{eq:Robin-problem}
			\begin{cases}
				-\Delta u=f & \text{in }\Omega,\\
				\partial_{\nu}u+\alpha u=g
				& \text{on }\partial\Omega,
			\end{cases}
		\end{equation}
		where $f\in L^{n}(\Omega)$ and
		$g\in L^{\infty}(\partial\Omega)$. Then
		\begin{equation}
			\label{eq:ABP-Robin-Linfty}
			\|u\|_{L^{\infty}(\Omega)}
			\leq
			\frac{1}{\alpha}
			\|g\|_{L^{\infty}(\partial\Omega)}
			+C_{n}
			\Big(
			\operatorname{diam}(\Omega) +\frac{1}{\alpha}
			\Big)
			\|f\|_{L^{n}(\Omega)}.
		\end{equation}
	\end{lemma}
    \begin{remark}
    One can see \cite{Lieberman2013} for more general ABP estimate for solutions to uniform elliptic PDEs with oblique derivative problem. 
    \end{remark}
    \begin{proof}
Set
$
G=\|g\|_{L^\infty(\partial\Omega)}
$ and
$v=u-\frac{G}{\alpha}$.
Then
\[
-\Delta v=f
\;\;\text{in }\Omega\quad
\text{and}\quad
\partial_\nu v+\alpha v
=g-G\leq0
\;\;\text{on }\partial\Omega.
\]

Let
$
M_0 =\max_{\overline{\Omega}}v.$
If $M_0\leq0$, we have $\sup_{\Omega} u\le \frac{G}{\alpha}$. Assume that $M_0 >0$, and choose
$x_0\in\overline{\Omega}$ such that
$
v(x_0)=M_0.$
For any
$
0<r<
\frac{M_0}{\operatorname{diam}(\Omega)+\alpha^{-1}}
$
and any $p\in B_r(0)$, let $y_0\in\overline{\Omega}$ be a maximum point of
$\widetilde v(x):= v(x)- \langle p, x-x_0\rangle
$. In particular,
$
\widetilde v(y_0)\geq v(x_0)=M_0
$
and then
\begin{align*}
v(y_0) \geq M_0-|p|\,|y_0-x_0| 
> M_0 -\operatorname{diam}(\Omega)r > \frac{r}{\alpha}
\geq \frac{|p|}{\alpha}.
\end{align*}

We claim that $y_0\in\Omega$. Indeed, if $y_0\in\partial\Omega$, 
\[
\widetilde v_{\nu}(y_0)=\partial_\nu v(y_0)-p\cdot\nu(y_0)\geq0.
\]
On the other hand, the Robin boundary condition gives $\partial_\nu v(y_0)\leq-\alpha v(y_0)$. Hence
\[
0
\leq
\partial_\nu v(y_0)-p\cdot\nu(y_0)
\leq
-\alpha v(y_0)+|p|
<0,
\]
which is a contradiction. Thus $y_0\in\Omega$, then we have
\[
D v(y_0)=p
\quad\text{and}\quad
D^2v(y_0)\leq0.
\]
Let
\[
\Gamma^{+}
=
\left\{
y\in\Omega:
v(x)\leq
v(y)+ \langle Dv(y), x-y\rangle
\ \text{for all }x\in\overline{\Omega}
\right\}
\]
be the upper contact set of $v$. The preceding argument shows that
\[
B_r(0)\subset Dv(\Gamma^{+}).
\]
Hence, by the area formula,
\[
|B_1|r^n
\leq
\int_{\Gamma^{+}}
\det(-D^2v)\,dx.
\]
Since $D^2v\leq0$ on $\Gamma^{+}$, by the arithmetic--geometric mean inequality,
\[
\det(-D^2v)
\leq
\Big(\frac{-\Delta v}{n}\Big)^n
=
\Big(\frac{f}{n}\Big)^n
\qquad\text{on }\Gamma^{+}.
\]
It follows that
\[
|B_1|r^n
\leq
\frac{1}{n^n}
\int_{\Gamma^{+}}|f|^n\,dx
\leq
\frac{1}{n^n}
\|f\|_{L^n(\Omega)}^n.
\]
Let 
$
r$ tend to $
\frac{M_0}{\operatorname{diam}(\Omega)+\alpha^{-1}},
$
we obtain
\[
M_0
\leq
C_n
\left(
\operatorname{diam}(\Omega)+\frac{1}{\alpha}
\right)
\|f\|_{L^n(\Omega)},
\]
where $
C_n=\frac{1}{n|B_1|^{1/n}}.$
Therefore,
\[
\sup_\Omega u
\leq
\frac{1}{\alpha}
\|g\|_{L^\infty(\partial\Omega)}
+
C_n
\left(
\operatorname{diam}(\Omega)+\frac{1}{\alpha}
\right)
\|f\|_{L^n(\Omega)}.
\]

Applying the same argument to $-u$ gives
\[
-\inf_\Omega u
\leq
\frac{1}{\alpha}
\|g\|_{L^\infty(\partial\Omega)}
+
C_n
\left(
\operatorname{diam}(\Omega)+\frac{1}{\alpha}
\right)
\|f\|_{L^n(\Omega)}.
\]
Combining the last two inequalities yields
\eqref{eq:ABP-Robin-Linfty}.
\end{proof}

Based on the above lemmas, we immediately 
obtain the $C^0$-decay estimate for $v^{E}$.
	\begin{lemma}\label{lem:eigen-c0}
		There is  positive constant $C_{E,0}$, depending only on the
		$C^3$ bound of $\partial\Omega$, $\lambda^D$, and $g_D^{-1}$,
		such that, for all $\alpha$ big enough,
		\[
		\|v^{E,\alpha}\|_{L^\infty(\Omega)}
		\leq C_{E,0}\alpha^{-1}.
		\]
	\end{lemma}

		\begin{proof}
		Applying the ABP estimate \eqref{eq:ABP-Robin-Linfty} to
		\eqref{eq:eigen-difference}, we see that, for $\alpha \ge 1$,
		\begin{align}
        \label{estimate-c0-vEa}
		\|v^{E,\alpha}\|_{L^\infty}
		\leq \alpha^{-1}\|q^E\|_{L^\infty(\partial\Omega)}
		+C\|\lambda^\alpha v^{E,\alpha}+f^\alpha\|_{L^n(\Omega)} \leq  C\alpha^{-1}+\|v^{E,\alpha}\|_{L^n(\Omega)}.
		\end{align}
		By Lemma~\ref{lem:eigen-l2},
		\begin{align*}
		\|v^{E,\alpha}\|_{L^n(\Omega)}
		\leq
		\|v^{E,\alpha}\|_{L^\infty(\Omega)}^{\frac{n-2}{n}}
		\|v^{E,\alpha}\|_{L^2(\Omega)}^\frac{2}{n}
		\le C^{-\frac{2}{n}}\alpha^{-\frac{2}{n}}\|v^{E,\alpha}\|_{L^\infty(\Omega)}^\frac{n-2}{n}.
		\end{align*}
	Then we get the desired $L^{\infty}$ decay estimate by \eqref{estimate-c0-vEa} and Young's inequality.
	\end{proof}

	\subsection{The gradient estimate}
	Let $u^1$  be the solution of the following equation
	\begin{equation}\label{eq:eigen-u1}
		\begin{cases}
		-\Delta u^1= \alpha(\lambda^\alpha v^{E,\alpha}+f^\alpha)
		& \text{in }\Omega,\\
		u^1=q^E & \text{on }\partial\Omega,
		\end{cases}
	\end{equation}
	and set
	\[
	w:=u^1-\alpha v^{E,\alpha}.
	\]
	Then $w$ solves
	\begin{equation}\label{eq:eigen-w}
		\begin{cases}
		\Delta w=0 & \text{in }\Omega,\\
		\partial_\nu w+\alpha w=\partial_\nu u^1
		& \text{on }\partial\Omega.
		\end{cases}
	\end{equation}
	Also,
	\[
	w=q^E-\alpha v^{E,\alpha}
	=\partial_\nu v^{E,\alpha}\quad\text{on }\partial\Omega.
	\]
	The estimate \eqref{eq:eigenvalue-decay}, Lemma \ref{lem:eigen-c0} and the $C^{2,\beta}$ bound for $q^E$ imply that
	$\|Du^1\|_{L^\infty(\Omega)}\leq C$.
	The boundary maximum argument applied to \eqref{eq:eigen-w} then gives
	\[
	\|w\|_{L^\infty(\Omega)}
	\leq\alpha^{-1}\|\partial_\nu u^1\|_{L^\infty(\partial\Omega)}
	\leq C\alpha^{-1}.
	\]

	Next, we prove the decay estimates for $Dv^{E,\alpha}$ and $Dw$. Differentiating the equation
	for $v^{E,\alpha}$ gives, for $1\leq i,j\leq n$,
	\begin{equation}\label{eq:eigen-differentiated-interior}
		-\Delta v^{E,\alpha}_i
		=\lambda^\alpha v^{E,\alpha}_i+f^\alpha_i\;\; \mbox{and}\;\; 
		-\Delta v^{E,\alpha}_{ij}
		=\lambda^\alpha v^{E,\alpha}_{ij}+f^\alpha_{ij}
		\quad\text{in }\Omega.
	\end{equation}
	Let $\tau$ be a unit tangent vector to $\partial\Omega$, the tangential differentiation of $w=q^E-\alpha v^{E,\alpha}=\partial_\nu v^{E,\alpha}$ on $\partial\Omega$ gives
	\[
		w_\tau = v^{E,\alpha}_{\nu\tau}
		+ \langle Dv^{E,\alpha}, D_\tau\nu\rangle, 
    \quad {\rm i.e. } \quad
        v^{E,\alpha}_{\nu\tau}
		+ \langle Dv^{E,\alpha}, D_\tau\nu \rangle
		+ \alpha v^{E,\alpha}_\tau
		= \partial_\tau q^E.
	\]
	We will use these formulas in the boundary estimates below.
	\begin{lemma}[Gradient estimates for $v^{E,\alpha}$]\phantomsection
		\label{lem:eigen-c1}
		There are positive constants
			$C_{E,1.1}$ depending on the $C^3$ bound of $\partial\Omega$, $g_D^{-1}$; and $C_{E,1.2}$, depending on the $C^{3,1}$
			bound for $\partial\Omega$, $\lambda^D$, $g_D^{-1}$; such that for large $\alpha$
			\[
			\|Dv^{E,\alpha}\|_{L^\infty(\Omega)}
			\leq C_{E,1.1}\alpha^{-1}
			\quad \|Dw\|_{L^\infty(\Omega)}
			\leq C_{E,1.2}\alpha^{-1}.
			\]
	\end{lemma}

	\begin{proof}
		As before, we use the extension function $h$ of the boundary distance and set
		\[
		G(x,\xi)=v^{E,\alpha}_\xi
		{\color{red} + } \langle\xi,Dh\rangle w+\frac12w^2.
		\]
		If $G$ attains its maximum at $(x_0,\xi_0)$ with $x_0\in\Omega$, then
		\[
		\begin{aligned}
		\Delta G(x_0, \xi_0)
		={}&-\lambda^\alpha v^{E,\alpha}_{\xi_0}
		-f^\alpha_{\xi_0}
    {+ } \langle\xi_0,D\Delta h\rangle w  {+ } 2\langle \xi_0,D^2h Dw\rangle +|Dw|^2.
		\end{aligned}
		\]
		Using $Dv^{E,\alpha}=\alpha^{-1}(Du^1-Dw)$, Lemma \ref{lem:eigen-c0} and the sequel uniform estimate for $w$, we obtain
		\[
		0\geq\Delta G(x_0,\xi_0)
		\geq |Dw(x_0)|^2-C|Dw(x_0)|-C.
		\]
		Thus $|Dw(x_0)|\leq C$ and $G(x_0,\xi_0)\leq C\alpha^{-1}$.
		
        If now $x_0\in\partial\Omega$, the tangential and nontangential arguments
		in the proof of Lemma~\ref{lem:torsion-c1} hold still without
		change, since
		$\partial_\nu v^{E,\alpha}=w$ and
		$\partial_\nu w+\alpha w=\partial_\nu u^1$ on $\partial\Omega$.
		Applying the same argument to $(-v^{E,\alpha},-u^1,-w)$ gives
		\[
		\|Dv^{E,\alpha}\|_{L^\infty(\Omega)}\leq C\alpha^{-1}.
		\]

		Moreover, for $Dw$, we set
		\[
		F^\alpha:=\alpha(\lambda^\alpha v^{E,\alpha}+f^\alpha)
		\]
		so that
		\[
		\|F^\alpha\|_{L^\infty(\Omega)}
		+\|DF^\alpha\|_{L^\infty(\Omega)}\leq C.
		\]
		Let $\widetilde u$ solve
		\[
		\begin{cases}
		\Delta\widetilde u=0 & \text{in }\Omega,\\
		\widetilde u=\partial_\nu u^1 & \text{on }\partial\Omega,
		\end{cases}
		\]
		and denote $\widetilde w:=\widetilde u-\alpha w$. The uniform bound for
		$F^\alpha$, the preceding $C^{2,\beta}$ bound for $q^E$, and standard
		elliptic regularity theory, applied successively to the Dirichlet problems for
		$u^1$ and $\widetilde u$, give
		\[
		\|D^2u^1\|_{L^\infty(\Omega)}
		+\|D\widetilde u\|_{L^\infty(\overline\Omega)}\leq C.
		\]
		Moreover, $\widetilde w=\partial_\nu w$ on $\partial\Omega$ and
		\[
		\begin{cases}
		\Delta\widetilde w=0 & \text{in }\Omega,\\
		\partial_\nu\widetilde w+\alpha\widetilde w
		=\partial_\nu\widetilde u & \text{on }\partial\Omega,
		\end{cases}
		\]
		The boundary maximum argument gives $\|\widetilde w\|_{L^\infty(\Omega)}\leq C\alpha^{-1}$. Consider the auxiliary function
		\[
		\widetilde G(x,\xi)
		=w_\xi +\langle\xi,Dh\rangle\widetilde w
		+\frac{\widetilde w^2}{2}.
		\]
		At an interior maximum, recalling that $w$ and $\widetilde w$ are harmonic,
		\[
		0 \ge \Delta\widetilde G
		=  \langle\xi,D\Delta h\rangle\widetilde w
		{+ } 2\langle\xi,D^2h D\widetilde w\rangle
		+|D\widetilde w|^2,
		\]
		so $|D\widetilde w|\leq C$ there. Since
		$Dw=\alpha^{-1}(D\widetilde u-D\widetilde w)$, the interior maximum is
		bounded by $C\alpha^{-1}$. 
        
        The boundary argument is exactly the one in the proof of
		Lemma~\ref{lem:torsion-c1}. Applying it also to
		$(-w,-\widetilde w)$ yields $\|Dw\|_{L^\infty(\Omega)}\leq C\alpha^{-1}$.
	\end{proof}

	The proof of
	Lemma~\ref{lem:eigen-c1} also gives
	\begin{equation}\label{eq:eigen-u1-c2}
		\|D^2u^1\|_{L^\infty(\Omega)}\leq C.
	\end{equation}

	\subsection{The second order derivative estimate}
	Again, we will reduce the global second-order derivative estimate to the double
	normal derivative on the boundary, as before, convexity is not necessary in this step.
	\begin{lemma}\label{lem:eigen-hessian-reduction}
		There is $C>0$, depending
		only on the $C^{3,1}$ bound for $\partial\Omega$,
		$\lambda^D$ and $g_D^{-1}$, such that, for every large
		$\alpha$,
		\[
		\|D^2v^{E,\alpha}\|_{L^\infty(\Omega)}
		\leq C\Big(\alpha^{-1}
		+\max_{\partial\Omega}|v^{E,\alpha}_{\nu\nu}|\Big).
		\]
	\end{lemma}

	\begin{proof}
		Let 
		\[
		W(x,\xi)= v^{E,\alpha}_{\xi\xi}
		{+ } 2\langle\xi,Dh\rangle
		\langle\widetilde\xi,Dw { + }D^2h Dv^{E,\alpha}\rangle
		+\frac{|Dw|^2}{2},
		\]
		where $\widetilde\xi=\xi-\langle\xi,Dh\rangle Dh$.
		Let $(x_0,\xi_0)$ be a maximum point of $W$ in $\overline\Omega\times {\mathbb S}^{n-1}$. For $\alpha$ large, since 
		$W-v^{E,\alpha}_{\xi\xi}$ are $O(\alpha^{-1})$ by preceding gradient estimate, we may
		assume
		\[
		v^{E,\alpha}_{\xi_0\xi_0}(x_0)\geq\alpha^{-1},
		\]
		otherwise we have $W(x_0, \xi_0) \le C\alpha^{-1}$. On the other hand, 
		\begin{align}
        \label{estimate-LvEa}
		|\Delta v^{E,\alpha}|
		=|\lambda^\alpha v^{E,\alpha}+f^\alpha|
		\leq C\alpha^{-1}.  
		\end{align}
		For an orthonormal basis $\{e_i\}_{i=1}^n$, the maximality at $(x_0, \xi_0)$ gives
		\[
		v^{E,\alpha}_{ii}(x_0)
		\leq v^{E,\alpha}_{\xi_0\xi_0}(x_0)+C\alpha^{-1}.
		\]
		As before, using the trace bound \eqref{estimate-LvEa}, we can claim
		\begin{equation}\label{eq:eigen-hessian-comparison}
		|D^2v^{E,\alpha}(x_0)|
		\leq C_n v^{E,\alpha}_{\xi_0\xi_0}(x_0)
		+C\alpha^{-1}.
		\end{equation}

		\textbf{Case 1: $x_0\in \Omega$.} By
		Lemma~\ref{lem:eigen-c1} and \eqref{estimate-LvEa}, 
		\[
		\begin{aligned}
		0\geq\Delta W(x_0,\xi_0)
		&\geq |D^2w(x_0)|^2
		-C|D^2w(x_0)| -C|D^2v^{E,\alpha}(x_0)|-C.
		\end{aligned}
		\]
		As for Lemma~\ref{lem:torsion-hessian-reduction}, this calculation can be
		justified by local mollification to claim that the coefficients $C$ are controlled by $C^{3,1}$ bound of $\partial\Omega$.
		The identity
		\[
		D^2w=D^2u^1-\alpha D^2v^{E,\alpha}
		\]
		and the uniform bound \eqref{eq:eigen-u1-c2} allow to have
		\[
		0\geq-C+\frac{\alpha^2}{4}
		|D^2v^{E,\alpha}(x_0)|^2,
		\]
		hence $|D^2v^{E,\alpha}(x_0)|\leq C\alpha^{-1}$.

		\textbf{Case 2: $x_0\in\partial \Omega$ and $\xi_0$ is non-tangential at $x_0$.}
		Write
		\[
		\xi_0=a\tau+b\nu(x_0),\quad a=\sqrt{1-b^2} < 1,
		\]
		with $\tau \in {\mathbb S}^{n-1}$ tangent to $\partial \Omega$ at $x_0$. Differentiating
		$\partial_\nu v^{E,\alpha}=w$ in the $\tau$ direction gives
		\[
		v^{E,\alpha}_{\nu\tau}
		+v^{E,\alpha}_k\nu_\tau^k=w_\tau.
		\]
		Using this identity and $W(x_0,\tau)\leq W(x_0,\xi_0)$, we can proceed exactly as for Lemma \ref{lem:torsion-hessian-reduction} and conclude
		\[
		W(x_0,\xi_0)
		\leq C\bigl(\alpha^{-1}+
		\max_{\partial\Omega}|v^{E,\alpha}_{\nu\nu}|\bigr).
		\]

		\textbf{Case 3: $x_0\in\partial \Omega$ and $\xi_0$ is tangential at $x_0$.}
		As before, we assume that $x_0=0$, $\xi_0=e_1$, $\nu(x_0)=-e_n$, and write
		the boundary locally as $x_n=\rho(x')$. Differentiating
		$Dv^{E,\alpha}\cdot\nu=w$ twice in the $x_1$ direction, and
		using $w=u^1-\alpha v^{E,\alpha}$, yields
		\[
		v^{E,\alpha}_{n11}
		=\alpha v^{E,\alpha}_{11}
		+O\bigl(1+|D^2v^{E,\alpha}|\bigr),
		\]
		where $C$ depends on the $C^3$ bound of $\partial\Omega$ because the
		calculation contains $D_\tau^2\nu$. Since 
		\[
		0 \ge W_n(x_0,e_1)
		=v^{E,\alpha}_{11n}
		+O\bigl(|Dv^{E,\alpha}|+|Dw|\bigr)+ \sum_{k=1}^n w_kw_{kn}.
		\]
		Using Lemma~\ref{lem:eigen-c1},
		$D^2w=D^2u^1-\alpha D^2v^{E,\alpha}$, and
		\eqref{eq:eigen-hessian-comparison}, we obtain, as for Lemma \ref{lem:torsion-hessian-reduction},
		\[
		0\geq W_n(x_0,e_1)
		\geq(\alpha-C)v^{E,\alpha}_{11}(x_0)-C.
		\]
		This yields, in any case $v^{E,\alpha}_{11}(x_0)\leq C\alpha^{-1}$.

		All three cases give an upper bound for every directional second
		derivative. The trace bound then gives the corresponding lower
		bound, and the stated second-order derivative estimate follows.
	\end{proof}

	Finally, we estimate the double normal derivatives of $v^{E,\alpha}$ on the boundary.
	\begin{lemma}[Eigenfunction double normal estimate]
	\label{lem:eigen-double-normal}
		There is
		$C>0$, depending only on the $C^{3,1}$ bound for
		$\partial\Omega$, $\lambda^D$, and $g_D^{-1}$, such that
		\[
		\max_{\partial\Omega}|v^{E,\alpha}_{\nu\nu}|
		\leq C\alpha^{-1}
		\quad \mbox{for all $\alpha$ sufficiently large}.
		\]
	\end{lemma}

	\begin{proof}
		Use the fixed strip $\Omega_{\mu_*}$ and the function $h$
		introduced before. Set
		\[
		M_E:=\max_{\partial\Omega}|v^{E,\alpha}_{\nu\nu}|,
		\quad
		\varphi:=\langle Dv^{E,\alpha},Dh\rangle+w.
		\]
		We have $\varphi=0$ on
		$\partial\Omega$, since $Dh=-\nu$ and $w=\partial_\nu v^{E,\alpha}$. Moreover,
		\[
		\Delta\varphi
		=-\langle\lambda^\alpha Dv^{E,\alpha}+Df^\alpha,Dh\rangle
		+2\langle D^2v^{E,\alpha}, D^2h\rangle 
		+ \langle Dv^{E,\alpha}, D(\Delta h)\rangle.
		\]
		By \eqref{eq:distance-normal-hessian}, $Dh_\nu =0$ on
		$\partial\Omega$. The preceding estimate for $w$, together with Lemmas~\ref{lem:eigen-l2},
		\ref{lem:eigen-c1} and \ref{lem:eigen-hessian-reduction}, gives
		\begin{equation}\label{eq:eigen-phi-distance-bounds}
		|\varphi|\leq C_2\alpha^{-1},
		\quad
		|\Delta\varphi|\leq C_2\left(M_E+\alpha^{-1}\right)
		\quad\text{in }\Omega_{\mu_*},
		\end{equation}
		and $|\partial_\nu w|\leq C_2\alpha^{-1}$ on $\partial\Omega$ with a uniform constant $C_2$. Fix now
		\[
		C_0\geq
		\|\Delta d\|_{L^\infty(\Omega_{\mu_*})}+8C_2+1,
		\quad
		\mu_0:=\min\left\{\mu_*,\frac{1}{2C_0}\right\}.
		\]
		Considering $g = h-\frac{C_0}{2}h^2$ in $\Omega_{\mu_0}$, proceeding exactly as for Lemma \ref{lem:torsion-double-normal} (with auxiliary functions $P_E = \pm\varphi-\frac{M_E}{4}g$), we arrive at
		\[
		M_E\leq\frac{16C_2}{3\mu_0}\alpha^{-1}.
		\]
        So we omit more details.
	\end{proof}

	Clearly, Lemmas \ref{lem:eigen-c0}--\ref{lem:eigen-double-normal} yield $\|v^{E,\alpha}\|_{C^2(\overline\Omega)}\leq C\alpha^{-1}$, hence the proof of Theorem \ref{thm:eigen-decay} is completed. \qed

\section{From decay estimates to concavity}\label{sec:concavity}
\reset

By the decay estimates obtained in Sections~2 and~3,
\begin{equation}\label{eq:C2-decay-both}
\|v^{T,\alpha}\|_{C^2(\overline\Omega)}
+
\|v^{E,\alpha}\|_{C^2(\overline\Omega)}
\leq C\alpha^{-1}.
\end{equation}

Once having \eqref{eq:C2-decay-both},
we see that $-\sqrt{u^{T,\alpha}}$ and $-\log u^{E,\alpha}$ are strictly convex in a fixed boundary strip.  The corresponding result was proved by Crasta and Fragal\`a \cite[Proposition~4.3]{CrastaFragala2021}, using an argument inspired by Korevaar \cite{Korevaar1983}. For completeness, we present the proof.
\begin{lemma}[Strict convexity near the boundary]
\label{prop:boundary-convexity}
Assume that $\Omega$ is uniformly convex of class $C^{3,1}$.
There exist $\mu_1>0$ and $\alpha_0>0$ such that, for every
$\alpha\geq\alpha_0$,
\begin{align}
\label{convexity-boundary}
D^2\bigl(-\sqrt{u^{T,\alpha}}\bigr)>0,\quad
D^2\bigl(-\log u^{E,\alpha}\bigr)>0
\quad\text{in }\Omega_{\mu_1}.
\end{align}
\end{lemma}

\begin{proof}
We first consider $u^{T,\alpha}$. For any
$y_0\in\partial\Omega$, choose principal coordinates centered at
$y_0$ such that
\[
x_n=\rho(x'),
\quad
\rho(0)=0,
\quad
\nabla\rho(0)=0,
\quad
\rho_{ij}(0)=\kappa_i\delta_{ij},
\quad
1\leq i,j\leq n-1.
\]
Since $u^{T,D}=0$ on $\partial\Omega$, there hold $u^{T,D}_{ij}(y_0)
=
-|Du^{T,D}(y_0)|\kappa_i\delta_{ij}$ for $1\leq i,j\leq n-1$. 
By the uniform convexity of $\Omega$ and the Hopf lemma, there exists
$\tau_0>0$ such that
\[
\inf_{y\in\partial\Omega}
\inf_{\xi\in\mathbb S^{n-1}\cap T_y\partial\Omega}
\left\langle
-D^2u^{T,D}(y)\xi,\xi
\right\rangle
\geq 8\tau_0.
\]
By continuity, after choosing $\mu_1>0$ sufficiently small, for every
$x_0\in\Omega_{\mu_1}$ and the nearest point
$y_0\in\partial\Omega$ to $x_0$,
\[
\inf_{\xi\in\mathbb S^{n-1}\cap T_{y_0}\partial\Omega}
\left\langle
-D^2u^{T,D}(x_0)\xi,\xi
\right\rangle
\geq 4\tau_0.
\]
By \eqref{eq:C2-decay-both}, increasing $\alpha_0$ if necessary, we have, for all $\alpha\geq\alpha_0$,
\begin{equation}\label{eq:tangential-negative}
\inf_{\xi\in\mathbb S^{n-1}\cap T_{y_0}\partial\Omega}
\left\langle
-D^2u^{T,\alpha}(x_0)\xi,\xi
\right\rangle
\geq 2\tau_0.
\end{equation}

Consider now $\eta\in\mathbb S^{n-1}$. We distinguish two cases.

\textbf{Case 1:}
$
|\langle\eta,\nu(y_0)\rangle|\leq\epsilon_0.
$ Here $\epsilon_0$ is a positive small constant independent of $x_0$, to be precised below. 

Decompose $\eta$ as follows
\[
\eta=a\eta'+b\nu(y_0),
\quad
\eta'\in\mathbb S^{n-1}\cap T_{y_0}\partial\Omega,
\quad
a=\sqrt{1-b^2}.
\]
Since $\|D^2u^{T,\alpha}\|_{L^\infty(\Omega)}\leq C$ uniformly for
$\alpha$ large,
\[
\left|
\left\langle D^2u^{T,\alpha}(x_0)\eta,\eta\right\rangle
-
\left\langle D^2u^{T,\alpha}(x_0)\eta',\eta'\right\rangle
\right|
\leq C\epsilon_0.
\]
Choosing $\epsilon_0>0$ sufficiently small and using
\eqref{eq:tangential-negative}, we obtain
\[
-\left\langle
D^2u^{T,\alpha}(x_0)\eta,\eta
\right\rangle
\geq \tau_0.
\]
Hence
\[
\begin{aligned}
\left\langle
D^2\bigl(-\sqrt{u^{T,\alpha}}\bigr)(x_0)\eta,\eta
\right\rangle
&=
-\frac{
\left\langle D^2u^{T,\alpha}(x_0)\eta,\eta\right\rangle}
{2\sqrt{u^{T,\alpha}(x_0)}}\
+\ \frac{
|\langle Du^{T,\alpha}(x_0),\eta\rangle|^2}
{4\sqrt{u^{T,\alpha}(x_0)^3}}\\
&\geq
\frac{\tau_0}
{2\sqrt{u^{T,\alpha}(x_0)}}>0.
\end{aligned}
\]

\textbf{Case 2:}
$
|\langle\eta,\nu(y_0)\rangle|>\epsilon_0.
$
By the Dirichlet boundary condition and Hopf Lemma, we have $\min_{\partial\Omega}|\partial_\nu u^{T,D}| = C_0 > 0$. Therefore
\[
|\langle Du^{T,D}(y_0),\eta\rangle| = |\langle\eta,\nu(y_0)\rangle| \times |\partial_\nu u^{T,D}(y_0)| 
\geq
C_0\epsilon_0.
\]
Using \eqref{eq:C2-decay-both} and decreasing $\mu_1$ and increasing
$\alpha_0$ if necessary, we obtain, for $\alpha$ large,
\begin{equation*}
|\langle Du^{T,\alpha}(x_0),\eta\rangle|
\geq
\frac{C_0\epsilon_0}{2}.
\end{equation*}
Consequently,
\begin{align}
\label{Case2}
\left\langle
D^2\bigl(-\sqrt{u^{T,\alpha}}\bigr)(x_0)\eta,\eta
\right\rangle
&\geq
-\frac{C}{2\sqrt{u^{T,\alpha}(x_0)}}
+\frac{
\epsilon_0^2C_0^2}
{16\sqrt{u^{T,\alpha}(x_0))^3}}.
\end{align}
Since $u^{T,D}=0$ on $\partial\Omega$, $u^{T,\alpha}(x_0)
\leq
\|Du^{T,D}\|_{L^\infty(\Omega)}
h(x_0)
+
C\alpha^{-1}
\leq
C\mu_1+C\alpha^{-1}.
$
Thus, after decreasing $\mu_1$ and increasing $\alpha_0$ once more,
the right hand side of \eqref{Case2} becomes positive. 

\smallskip
As $\eta \in \mathbb{S}^{n-1}$ is arbitrary, the first claim in \eqref{convexity-boundary} holds true for $\mu_1 > 0$ small enough.

The eigenfunction is treated in the same way, we will go through quickly. Since
$u^{E,D}=0$ on $\partial\Omega$, uniform convexity and the Hopf lemma
give, after possibly decreasing $\mu_1$ and increasing $\alpha_0$, there are $\epsilon_0, c_0 \in (0, 1)$ such that 
\[
\begin{cases}
-\left\langle
D^2u^{E,\alpha}(x_0)\eta,\eta
\right\rangle
\geq c_0, & \mbox{ whenever }\; |\langle\eta,\nu(y_0)\rangle|\leq\epsilon_0;\\
|\langle Du^{E,\alpha}(x_0),\eta\rangle|\geq c_0, & \mbox{ whenever }\; |\langle\eta,\nu(y_0)\rangle|>\epsilon_0.\end{cases}
\]
Since
\[
\left\langle
D^2\bigl(-\log u^{E,\alpha}\bigr)\eta,\eta
\right\rangle
=
-\frac{
\left\langle D^2u^{E,\alpha}\eta,\eta\right\rangle}
{u^{E,\alpha}}
+
\frac{
|\langle Du^{E,\alpha},\eta\rangle|^2}
{(u^{E,\alpha})^2},
\]
the first case gives positivity directly. In the second case,
\[
\left\langle
D^2\bigl(-\log u^{E,\alpha}\bigr)(x_0)\eta,\eta
\right\rangle
\geq
-\frac{C}{u^{E,\alpha}(x_0)}
+
\frac{c_0^2}{u^{E,\alpha}(x_0)^2}.
\]
Moreover, $u^{E,\alpha}(x_0)
\leq
\|Du^{E,D}\|_{L^\infty(\Omega)}
d(x_0,\partial\Omega)
+
C\alpha^{-1}
\leq
C\mu_1+C\alpha^{-1}.$
Therefore, after decreasing $\mu_1$ and increasing $\alpha_0$ if
necessary, we get the second claim in \eqref{convexity-boundary}.
\end{proof}

\begin{proof}[Proof of Theorems~\ref{thm:torsion-concavity}
and~\ref{thm:eigen-concavity}]
Set
\[
K:=\overline{\Omega\setminus\Omega_{\mu_1}}.
\]
Then $K\subset\subset\Omega$. The Dirichlet concavity results imply
\[
D^2\bigl(-\sqrt{u^{T,D}}\bigr)>0,
\quad
D^2\bigl(-\log u^{E,D}\bigr)>0
\quad\text{in }\Omega.
\]
Hence, by compactness of $K$, there exists $c_0>0$ such that
\[
D^2\bigl(-\sqrt{u^{T,D}}\bigr)\geq c_0 I,
\quad
D^2\bigl(-\log u^{E,D}\bigr)\geq c_0 I
\quad\text{on }K.
\]
Moreover,
\[
\min_K u^{T,D}>c_1,
\quad
\min_K u^{E,D}>c_1.
\]
Therefore, by \eqref{eq:C2-decay-both}, for all sufficiently large $\alpha$
\[
D^2\bigl(-\sqrt{u^{T,\alpha}}\bigr)>0,
\quad
D^2\bigl(-\log u^{E,\alpha}\bigr)>0
\quad\text{on }K.
\]

Combining this with Lemma~\ref{prop:boundary-convexity}, we are done.
\end{proof}

\bigskip
\noindent
{\bf Acknowledgements.} Both authors are partially supported by Science and Technology Commission of Shanghai Municipality (No.~22DZ2229014). D.Y. is also supported by NSFC (No.~12671132).

\medskip
\noindent{\bf Data availability.}
Data sharing is not applicable as no new data were created or analyzed in the study.

\medskip
\noindent{\bf Conflict of interest statement.}
The authors declare that there is no conflict of interest relevant to
this article.

	\bibliographystyle{acm}
	\bibliography{Ye_Zhang2026}
	\enlargethispage{4\baselineskip}

\end{document}